\documentclass[12pt]{amsart}

\usepackage{amscd,mathrsfs,epsfig,amsthm,amssymb,amsmath}
\usepackage[all]{xy} 
\usepackage[T1]{CJKutf8}
\usepackage[sans]{dsfont}  
\usepackage{hyperref}
\usepackage{tabularx}
\usepackage{rotating}

\usepackage{graphicx} 
\usepackage{cite}
\usepackage{url}

\newtheorem{theorem}{Theorem}[subsection]
\newtheorem{corollary}[theorem]{Corollary}

\newtheorem{definition}[theorem]{Definition}

\newtheorem{remark}[theorem]{Remark}

\newtheorem*{thma}{Theorem A}
\newtheorem*{thmb}{Theorem B}

\allowdisplaybreaks[2]

\def\calC{{\mathcal C}}

\def\calJ{{\mathcal J}}

\def\calP{{\mathcal P}}

\def\Hom{\mathop{\rm Hom}\nolimits}

\def\lim{\mathop{\varinjlim}\nolimits}
 
\def\Ob{\mathop{\rm Ob}\nolimits} 
\def\Mor{\mathop{\rm Mor}\nolimits}

\DeclareMathOperator{\Ab}{\rm Ab}
\DeclareMathOperator{\rMod}{Mod-}

\DeclareMathOperator{\Add}{\rm Add}
\DeclareMathOperator{\dgCat}{\rm dg \mbox{-} Cat}
\DeclareMathOperator{\dgCh}{{\rm Ch}_{\it dg}}
\DeclareMathOperator{\Ch}{{\rm Ch}}
\DeclareMathOperator{\op}{\rm op}

\begin{document}

\title[modules over dg-representations of small categories]{a characterization of modules over dg-representations of small categories}

\author{Mawei Wu}
\address{School of Mathematics and Statistics, Lingnan Normal University, Zhanjiang, Guangdong 524048, China}
\email{wumawei@lingnan.edu.cn}





\begin{abstract}
Let $\calC$ be a small category and let $R$ be a dg-representation of the category $\calC$, that is, a pseudofunctor from a small category to the category of small dg $k$-categories, where $k$ is a commutative unital ring. In this paper, we mainly study the category $\rMod R$ of right modules over $R$. We characterize it as an ordinary category of dg-modules over a (differential graded) dg-category $Gr(R)$, where $Gr(R)$ is the linear Grothendieck construction of $R$. This characterization generalizes a result (\cite[Theorem 3.18]{EV17}) of Estrada and Virili to the dg-category context. Furthermore, as some applications of the main characterization theorem, we classify the hereditary torsion pairs, (split) TTF triples and Abelian recollements in $\rMod R$ respectively.
\end{abstract}

\maketitle

\tableofcontents

\section{Introduction}
Let $\calC$ be a small category and let $R: \calC \to \Add$ be a representation of the category $\calC$ which is a pseudofunctor from a small category $\calC$ to the category of small preadditive categories $\Add$. In \cite{EV17}, the authors study the category $\rMod R$ of right modules over $R$. They proved that the category $\rMod R$ is a Grothendieck category (see \cite[Theorem 3.18]{EV17}). Recently, we reproved this result using different method (see \cite[Theorem A]{Wu24a}). Roughly speaking, the category $\rMod R$ is like `gluing' all the right $R(i)$-modules together, where $R(i)$ is an additive category for each $i \in \Ob \calC$. Our main motivation is trying to replace the additive categories $R(i)$ by the dg-categories, hoping to obtain some analogous results as that of \cite{Wu24a}, so that we can `glue' all the right dg-modules together.

Motivating by the definitions in \cite{EV17}, we newly introduce the notions of \emph{the dg-representation of a small category} and \emph{the right modules over it}. Let $\calC$ be a small category and $k$ be a commutative unital ring, a dg-representation $R$ of the category $\calC$ is defined to be a pseudofunctor from a small category $\calC$ to the category of small dg $k$-categories $\dgCat_k$ (see Definition \ref{rep}). Given a dg-representation $R$ of the category, one can define the category $\rMod R$ of all right modules over $R$ (see Definition \ref{rmod}). Let $\dgCh(k)$ be \emph{the dg-category of chain complexes of $k$-modules}. All complexes considered in this paper are cochain complexes. Our first result is to characterize $\rMod R$ as an \emph{ordinary} category (\emph{not} dg-category) of dg-modules over $Gr(R)$, which generalizes results \cite[Theorem 3.18]{EV17} and \cite[Theorem A]{Wu24a} as well as \cite[Theorem 3.0.1]{Wu24}(or see \cite[Proposition 5]{How81}), to the dg-category context, where $Gr(R)$ is the linear Grothendieck construction of $R$ (see Definition \ref{lingrocon}).

\begin{thma} ([Theorem \ref{howeplus} and Corollary \ref{Groproj}])
Let $\calC$ be a small category and let $R: \calC \to \dgCat_k$ be a dg-representation of the category $\calC$, then we have the following equivalence
 $$
 \rMod R \simeq (Gr(R)^{\op}, \dgCh(k)).
 $$ 
Consequently, the category of right $R$-modules $\rMod R$ is a Grothendieck category and has a projective generator.     
\end{thma}

In the Theorem A above, the category $(Gr(R)^{\op}, \dgCh(k))$ is, though its homomorphisms set has a natural dg-structure, an ordinary category of right dg-modules over the dg-category $Gr(R)$. 

Using a result of Freyd (see \cite[Theorem 3.1]{Mit72}), we can characterize $\rMod R$ as an Ab-valued functor category: $\rMod R \simeq (\calP^{\op}, \Ab)$, for an explicit full subcategory $\calP$ of $\rMod R$ (see Corollary \ref{redab}). In this paper, we will denote by Ab the category of Abelian groups. Then we can classify the hereditary torsion pairs, (split) TTF triples and Abelian recollements in $\rMod R$, for the definitions of all these, one can see Section \ref{tta}.

\begin{thmb} ([Theorem \ref{htp}, \ref{ttf}, \ref{sttf} and \ref{ar}])
Let $\calC$ be a small category and let $R: \calC \to \dgCat_k$ be a dg-representation of the category $\calC$. Then the hereditary torsion pairs, (split) TTF triples and Abelian recollements of $\rMod R$ can be classified respectively.
\end{thmb}

This paper is organized as follows. In Section \ref{prelim}, we firstly recall some definitions about (preadditive) categories and the linear Grothendieck topologies. Then, the definitions of dg-representation of a small category and its right modules are introduced. Finally, the definitions of the linear Grothendieck constructions as well as the concepts of torsion pairs, TTF triples and Abelian recollements are recorded. In Section \ref{cha}, a characterization of the category of modules $\rMod R$ over a dg-representation of a small category will be given. By a theorem of Freyd, we can also characterize $\rMod R$ as an Ab-valued functor category. In Section \ref{app}, some applications of the main characterization theorem will be given. To be more specific, the hereditary torsion pairs, (split) TTF triples and Abelian recollements of the category $\rMod R$ of modules over a dg-representation of a small category will be classified respectively.

\section{Preliminaries} \label{prelim}
In this section, we will firstly recall some definitions about (preadditive) categories and the linear Grothendieck topologies. Then, the definitions of dg-representation of a small category and its right modules will be newly introduced. Finally, the definitions of the linear Grothendieck constructions as well as the concepts of torsion pairs, TTF triples and Abelian recollements will be recorded.

\subsection{Categories and linear Grothendieck topologies}

\subsubsection{Categories}

In this subsection, we will recall some basic definitions about category theory, and mostly about the preadditive categories. This part is mostly taken from \cite[Section 2.1.1]{Wu24}, and we record it here for the  convenience of the readers on the one hand, and fixing the notations on the other hand.

\begin{definition}
Let $\calC$ be a small category. It is said to be \emph{finite} if $\Mor \calC$ is finite, and it is \emph{object-finite} if $\Ob \calC$ is finite.    
\end{definition}

\begin{definition}
An endomorphism $f \in \Hom_{\calC}(x, x)$ is called an \emph{idempotent endomorphsim} if it is an idempotent (i.e. $f^2=f$).   
\end{definition}

\begin{definition}
A category $\mathcal{A}$ is \emph{preadditive} if each morphism set $\Hom_{\mathcal{A}}(x, y)$ is endowed with the structure of an Abelian group.   \end{definition}

Given a small preadditive category $\mathcal{A}$, one can associate it with the following two categories: the \emph{additive closure} $\widehat{\mathcal{A}}$ and the \emph{Cauchy completion} $\widehat{\mathcal{A}}_{\oplus}$ of it (namely, the idempotent completion  of the additive closure of $\mathcal{A}$,  see \cite[Section 1.1 for their precise definitions]{PSV21}).

\begin{definition} (\cite[pages: 1148, 1150]{PSV21})
\begin{enumerate}    
\item  Let $\mathcal{A}$ be a small preadditive category, a \emph{(two-sided) ideal} of $\mathcal{A}$ is a subfunctor 
$$
\mathcal{I}(-, -) \leq \mathcal{A}(-, -): \mathcal{A}^{\rm op} \times \mathcal{A} \to \Ab. 
$$
That is, for any $f \in \Hom_{\mathcal{I}}(x, y)$ and $l \in \Hom_{\mathcal{A}}(y, y')$, $r \in \Hom_{\mathcal{A}}(x', x)$, $l \circ f \circ r \in \Hom_{\mathcal{I}}(x',y')$, which is a subgroup of $\Hom_{\mathcal{A}}(x',y')$. 

\item Let $\mathcal{I}, \mathcal{K}$ be two ideals of $\mathcal{A}$, the \emph{product ideal} $\mathcal{I} \cdot \mathcal{K}$ is defined as follows:
$$
(\mathcal{I} \cdot \mathcal{K})(x, y):=\left\{ \sum_{i=1}^{n} g_i \circ f_i \ |\ g_i \in \Hom_{\mathcal{I}}(z_i, y), f_i \in \Hom_{\mathcal{K}}(x, z_i)   \right\}.
$$

\item An ideal $\mathcal{I}$ is said to be \emph{idempotent} if $\mathcal{I} \cdot \mathcal{I}=\mathcal{I}$.
\end{enumerate}
\end{definition} 

Given a unitary ring, one can always construct a two-sided ideal generated by a given family of elements. The analogous construction in the preadditive categories is given as follows.

\begin{definition} (\cite[Definition 2.2]{PSV21})   
Let $f: x \to y$ be a morphism in $\mathcal{A}$ and let $\mathcal{M}$ be a set of morphisms of $\mathcal{A}$. Then
\begin{enumerate}
    \item \emph{the (two-sided) ideal of $\mathcal{A}$ generated by $f$}, denoted by $\mathcal{A}f\mathcal{A}: \mathcal{A}^{\rm op} \times \mathcal{A} \to \Ab$, is the subfunctor of $\mathcal{A}(-, -)$ such that $\mathcal{A}f\mathcal{A}(a, b)$ is the subgroup of $\mathcal{A}(a,b)$ generated by compositions $g \circ f \circ h$, where $h \in \mathcal{A}(a, x)$ and $g \in \mathcal{A}(y, b)$;
    \item \emph{the (two-sided) ideal of $\mathcal{A}$ generated by $\mathcal{M}$}, denoted by $\mathcal{A}\mathcal{M}\mathcal{A}: \mathcal{A}^{\rm op} \times \mathcal{A} \to \Ab$, is the sum
    $$
    \mathcal{A}\mathcal{M}\mathcal{A}:=\sum_{f \in \mathcal{M}} \mathcal{A}f\mathcal{A}.
    $$
That is, $\mathcal{A}\mathcal{M}\mathcal{A}(a, b)=\sum_{f \in \mathcal{M}} \mathcal{A}f\mathcal{A}(a, b)$, where the sum is the sum of subgroups of the Abelian group $\mathcal{A}(a, b)$. 
\end{enumerate}
\end{definition}

Now, let's recall the definition of the center of a preadditive category. 

\begin{definition} (\cite[page: 1147]{PSV21})   
Let $\mathcal{A}$ be a small preadditive category, the \emph{center} $Z(\mathcal{A})$ of $\mathcal{A}$ is the ring of self-natural transformations of the identity functor ${\rm Id}_{\mathcal{A}}$, that is,
$$
Z(\mathcal{A}):={\mathcal{F}un}(\mathcal{A}, \mathcal{A})({\rm Id}_{\mathcal{A}}, {\rm Id}_{\mathcal{A}}).
$$
where ${\mathcal{F}un}(\mathcal{A}, \mathcal{A})$ is the $\Ab$-enriched functor category.
\end{definition}

Let $\mathcal{A}$ be a small preadditive category. A \emph{right module $M$ over $\mathcal{A}$} is an (always additive) functor $M:\mathcal{A}^{\rm op} \to \Ab$. We denote by $(\mathcal{A}^{\op}, \Ab)$ \emph{the category of right $\mathcal{A}$-modules}. Given a class $\mathcal{S}$ of $\mathcal{A}$-modules and an $\mathcal{A}$-module $M$, one can construct a submodule ${\rm tr}_{\mathcal{S}}(M)$ of $M$ such that any map $S \to M$, with $S \in \mathcal{S}$, factors through the inclusion ${\rm tr}_{\mathcal{S}}(M) \to M$.
  
\begin{definition} (\cite[Definition 1.8]{PSV21})
Let $\mathcal{S}$ be a class of right $\mathcal{A}$-modules and $M$ a right $\mathcal{A}$-module, then the sum of the submodules of $M$ of the form ${\rm Im}(f)$, for some morphism $f: S \to M$ in $(\mathcal{A}^{\op}, \Ab)$, with $S \in \mathcal{S}$, is called \emph{the trace of $\mathcal{S}$ in $M$} and denoted by ${\rm tr}_{\mathcal{S}}(M)$.
\end{definition}

For a class $\mathcal{S}$ of right $\mathcal{A}$-modules, one can define a (two-sided) ideal ${\rm tr}_{\mathcal{S}}(\mathcal{A})$ of $\mathcal{A}$, which is called the trace of $\mathcal{S}$ in $\mathcal{A}$.

\begin{definition} (\cite[Section 1.3 and Section 2.2]{PSV21})
Let $\mathcal{A}$ be a preadditive category, $\mathcal{A}(-, -): \mathcal{A}^{\rm op} \times \mathcal{A} \to \Ab$ the regular $\mathcal{A}$-bimodule and $\mathcal{S}$ a class of right $\mathcal{A}$-modules. The assignment $x \mapsto {\rm tr}_{\mathcal{S}}(\mathcal{A}(-, x))$ defines a subfunctor of the functor $x \mapsto \mathcal{A}(-, x)$. Viewing this subfunctor as an $\mathcal{A}$-bimodule, then this $\mathcal{A}$-bimodule is called \emph{the trace of $\mathcal{S}$ in $\mathcal{A}$}, and denoted by ${\rm tr}_{\mathcal{S}}(\mathcal{A})$, it is a two-sided ideal of $\mathcal{A}$.
\end{definition}

\subsubsection{Linear Grothendieck topologies}
In this subsection, we will recall the definition of the linear Grothendieck topology on a preadditive category.  

\begin{definition} (\cite[Definition 2.1]{Low04}) \label{lintop}
 Let $\mathcal{A}$ be a preadditive category. A \emph{linear Grothendieck topology} on $\mathcal{A}$ is given by specifying, for every object $x$ in $\mathcal{A}$, a collection $\calJ(x)$ of subfunctors of $\Hom_{\mathcal{A}}(-, x)$ in $(\mathcal{A}^{\op}, \Ab)$ satisfying the following axioms:
 \begin{enumerate}
     \item $\Hom_{\mathcal{A}}(-, x) \in \calJ(x)$;
     \item for $S \in \calJ(x)$ and $f: y \to x$ in $\mathcal{A}$, the pullback $f^{*}S$ (i.e. the following diagram is a pullback) in $(\mathcal{A}^{\op}, \Ab)$ of $S$ along $f: \Hom_{\mathcal{A}}(-, y) \to \Hom_{\mathcal{A}}(-, x)$ is in $\calJ(y)$,
     $$
     \xymatrix{
      f^{*}S \ar[r] \ar[d] & S \ar[d] \\
     \Hom_{\mathcal{A}}(-, y) \ar[r]^{f}  & \Hom_{\mathcal{A}}(-, x); \\
     } 
     $$ 
     \item for $S_1 \in \calJ(x)$ and an arbitrary subfunctor $S_2$ of $\Hom_{\mathcal{A}}(-, x)$, if for every $f: y \to x \in S_1$, the pullback $f^{*}S_2$ is in $\calJ(y)$, it follows that $S_2$ is in $\calJ(x)$.
 \end{enumerate}
\end{definition}

\begin{remark}
If $\mathcal{A}$ is a single object category associated to a ring $k$, then the axioms in the Definition \ref{lintop} above correspond to those of a Gabriel topology on $k$ \cite{Gab62}.  
\end{remark}

\subsection{Modules over dg-representations of small categories}
In this subsection, the definitions of a dg-representation of a small category and its right modules will be introduced. For the theory of dg-categories and dg-modules, one can see \cite{Kel94, Kel06}. It is well known that a dg-category can be seen as a category enriched over the category of chain complexes. The following definition is the dg-version of the \cite[Definition 3.1]{EV17}.

\begin{definition} \label{rep}
Let $\calC$ be a small category, a \emph{dg-representation} of $\calC$ is a pseudofunctor $R: \calC \to \dgCat_k$, that is, $R$ consists of the following data:

\begin{enumerate}
    \item for each object $i \in \Ob \calC$, a dg-category $R(i)$;
    \item for all $i, j \in \Ob \calC$ and any morphism $a: i \to j$, a dg-functor $R(a): R(i) \to R(j)$;
    \item for each object $i \in \Ob \calC$, an isomorphism of dg-functors $\delta_i: 1_{R(i)} \overset{\sim}{\longrightarrow} R(1_i)$;
    \item for any pair of composable morphisms $a$ and $b$ in $\calC$, an isomorphism of dg-functors $\mu_{b,a}: R(b)R(a) \overset{\sim}{\longrightarrow} R(ba)$.    
\end{enumerate}
Furthermore, we suppose that the following axioms hold:

\begin{enumerate}
    \item[(Rep.1)] given three composable morphisms $i \overset{a}{\longrightarrow} j \overset{b}{\longrightarrow} k \overset{c}{\longrightarrow} h$ in $\calC$, the following diagram commutes
     $$
     \xymatrix @C=5pc {
      R(c)R(b)R(a) \ar[rr]^{R(c)\mu_{b,a}} \ar[d]_{\mu_{c,b} R(a)} & & R(c)R(ba) \ar[d]^{\mu_{c,ba}} \\
      R(cb)R(a) \ar[rr]^{\mu_{cb,a}}  & & R(cba), \\
     } 
     $$ 
    \item[(Rep.2)] given a morphism $(a: i \to j) \in \Mor \calC$, the following diagram commutes
     $$
     \xymatrix{
        & R(a) \ar[dl]_{R(a)\delta_i} \ar[dr]^{\delta_jR(a)} \ar@{=}[dd] & \\
   R(a)R(1_i) \ar[dr]_{\mu_{a,1_i}}     &  &  R(1_j)R(a) \ar[dl]^{\mu_{1_j,a}}  \\
      & R(a). &  \\
     } 
     $$ 
\end{enumerate}
\end{definition}

\begin{remark} \label{diff}
\begin{enumerate}
    \item Since the $\delta_i$ and $\mu_{b,a}$ are 2-isomorphisms 
    let's denote $\eta_i=(\delta_i)^{-1}$ and $\theta_{b,a}=(\mu_{b,a})^{-1}$ by their inverses respectively, namely
    $$
     \xymatrix @C=5pc {
     R(1_i) \ar@<.6ex>[r]^{\eta_i}  &  1_{R(i)} \ar@<.6ex>[l]^{\delta_i},  \\
     R(ba)  \ar@<.6ex>[r]^{\theta_{b,a}} &  R(b)R(a) \ar@<.6ex>[l]^{\mu_{b,a}}.   \\
     } 
     $$
     We will use these inverses to define $R$-modules (see Definition \ref{rmod}).
     \item Given a dg-representation $R: \calC \to \dgCat_k$ and a morphism $(a: i \to j) \in \Mor \calC$, we denote by 
     $$
     \xymatrix @C=5pc { 
     a_{!}: (R(i)^{\op}, \dgCh(k)) \ar@<.6ex>[r] & (R(j)^{\op}, \dgCh(k)) :a^{*} \ar@<.6ex>[l] \\
     }
     $$
     the Quillen adjunction $(a_{!}, a^{*})$ induced by $R(a)$ (see \cite[page 627]{Toe07}).
     \item Comparing to \cite[Definition 3.1]{EV17} (or \cite[Definition 2.1.1]{Wu24a}), one of the main difference is that, in condition $(1)$, we replace additive categories by dg-categories.
\end{enumerate}  
\end{remark}

Given a dg-representation of a small category, one can define its right modules.

\begin{definition} \label{rmod}
Let $R: \calC \to \dgCat_k$ be a dg-representation of the small category $\calC$. A \emph{right $R$-module} $M$ consists of the following data:

\begin{enumerate}
    \item for all $i \in \Ob \calC$, a right dg-$R(i)$-module $M_i: R(i)^{\op} \to \dgCh(k)$;
    \item for any morphism $a: i \to j$ in $\calC$, a homomorphism $M(a): a^*M_j \to M_i$. 
\end{enumerate}
Furthermore, we suppose that the following axioms hold:

\begin{enumerate}
    \item[(Mod.1)] given two morphisms $a: i \to j$ and $b: j \to k$ in $\calC$, the following diagram commutes:
     $$
     \xymatrix @C=5pc {
      a^*b^*M_k \ar[d]_{\theta_{b,a}1_{M_k}} \ar[r]^{a^*M(b)} & a^*M_j \ar[r]^{M(a)} & M_i \\
     (ba)^*M_k \ar[urr]_{M(ba)} & & \\
     } 
     $$ 
    \item[(Mod.2)] for all $i \in \Ob \calC$, the following diagram commutes:
    $$
    \xymatrix @C=5pc {
      (1_i)^*M_i \ar[r]^-{M(1_i)}  & M_i \\
      M_i \ar[u]^{\eta_i 1_{M_i}} \ar[ur]_{1_{M_i}}  & \\
     } 
    $$
\end{enumerate}
\end{definition}

For a dg-representation $R: \calC \to \dgCat_k$ of the small category $\calC$, we will denote \emph{the category of all right $R$-modules} by $\rMod R$, which is the main object studied in this paper.

\subsection{Linear Grothendieck constructions}
Let $\calC$ be a small category and let $R:\calC \to \dgCat_k$ be a dg-representation of $\calC$, since $R$ is a pseudofunctor, which can be viewed as a colax functor, therefore, by \cite[pages: 21-22]{AP22}, one can associate to $R$ a linear Grothendieck construction.  

\begin{definition} \label{lingrocon}
 Let $\calC$ be a small category and let $R:\calC \to \dgCat_k$ be a dg-representation of $\calC$. Then a category $Gr(R)$, called the \emph{linear Grothendieck construction} of $R$, is defined as follows:
 \begin{enumerate}
     \item $\Ob Gr(R):=\cup_{i \in \Ob \calC}\{i\} \times \Ob R(i)=\{~_ix:=(i,x) ~|~ i \in \Ob \calC, x \in \Ob R(i)\}$;
     \item for each $_ix, _jy \in \Ob Gr(R)$, we set
     $$
     Gr(R)(_ix, _jy):=\bigoplus_{a \in \calC(i,j)}R(j)(R(a)x,y)=\bigoplus_{a \in \calC(i,j)}\bigoplus_{n \in \mathbb{Z}}R(j)^n(R(a)x,y),
     $$
     where $R(j)(R(a)x,y)$ is a dg-$k$-module;
     \item for each $_ix, _jy, _kz \in \Ob Gr(R)$ and each $f=(f_a^p)_{a \in \calC(i,j), p \in \mathbb{Z}}\in Gr(R)(_ix, _jy)$, $g=(g_b^q)_{b \in \calC(j,k), q \in \mathbb{Z}}\in Gr(R)(_jy, _kz)$, we set 
    \begin{align*}
	  g \circ f :=  & \left( \sum_{\substack{a \in \calC(i, j)\\ b \in \calC(j, k)\\ c=ba}} \sum_{p, r \in \mathbb{Z}} (-1)^{(n-r)r} g_b^{n-r-p} \circ (R(b)f_a)^p \circ (\theta_{b,a}x)^r \right)_{c \in \calC(i, k), n \in \mathbb{Z}}\\
    \end{align*}
     \item for each $_ix \in \Ob Gr(R)$ the identity $1_{_ix}$ is given by
     $$
     1_{_ix}:=(\delta_{a,1_i} \eta_ix)_{a \in \calC(i, i)} \in \bigoplus_{a \in \calC(i,i)}R(i)(R(a)x,x) = \bigoplus_{a \in \calC(i,i)} \bigoplus_{p \in \mathbb{Z}} R(i)^p(R(a)x,x),
     $$
     where $\eta_ix: R(1_i)x \to 1_{R(i)}x=x$, and $\delta_{a,1_i}$ is the Kronecker delta (\emph{not} a 2-isomorphism $\delta_i$ in a pseudofunctor!), that is, the $a$-th component of $1_{_ix}$ is $\eta_ix$ if $a=1_i$, and $0$ otherwise. 
 \end{enumerate}
\end{definition}

\subsection{Torsion pairs, TTF triples and Abelian recollements} \label{tta}
In this subsection, the definitions of torsion pairs, TTF triples and Abelian recollements will be recalled. This part is taken from \cite[Section 2.3]{Wu24}, and we record it here for the  convenience of the readers. 

\subsubsection{Torsion pairs}
Torsion theories (also called torsion pairs) were introduced by Dickson \cite{Dic66} in general setting of Abelian categories, taking as a model the classical theory of torsion Abelian groups. Torsion pairs have played an important role in studying the Grothendieck categories and their localizations.

\begin{definition}  \label{torpairdef} 
Let $\mathcal{A}$ be an Abelian category. A \emph{torsion pair} (or \emph{torsion theory}) in $\mathcal{A}$ is a pair of full subcategories $(\mathcal{X}, \mathcal{Y})$ such that:

\begin{enumerate}
    \item $\mathcal{X}=^{\perp}\mathcal{Y}$ and $\mathcal{Y}=\mathcal{X}^{\perp}$, where
    $$
    ^{\perp}\mathcal{Y}:=\{ a \in \Ob \mathcal{A}\ |\ \Hom_{\mathcal{A}}(a, y)=0, \forall\ y \in \Ob \mathcal{Y} \},
    $$
    $$
    \mathcal{X}^{\perp}:=\{ a \in \Ob \mathcal{A}\ |\ \Hom_{\mathcal{A}}(x, a)=0, \forall\ x \in \Ob \mathcal{X} \}.
    $$
    \item for each $a \in \Ob \mathcal{A}$, there is an exact sequence 
    \begin{align}\label{sqe}
    0 \to x_a \to a \to y_a \to 0, \tag{\dag}
    \end{align}
    with $x_a \in \Ob \mathcal{X}$ and $y_a \in \Ob \mathcal{Y}$.
\end{enumerate}
\end{definition}

If $(\mathcal{X}, \mathcal{Y})$ is a torsion pair in $\mathcal{A}$, then $\mathcal{X}$ is called a \emph{torsion class} and $\mathcal{Y}$ is called a \emph{torsion-free class}. A subcategory of $\mathcal{A}$ is the torsion class (resp. torsion-free class) of some torsion pair if and only if it is closed under quotients, direct sums and extensions (resp. subobjects, direct products and extensions, respectively), see \cite[Theorem 2.3]{Dic66}. 

There are some special kinds of torsion pairs will be considered in this paper.

\begin{definition}
\begin{enumerate}
    \item A torsion pair $(\mathcal{X}, \mathcal{Y})$ in $\mathcal{A}$ is called \emph{hereditary} if $\mathcal{X}$ is also closed under subobjects.
    \item A torsion pair $(\mathcal{X}, \mathcal{Y})$ is said to be \emph{split} if, for any object $x \in \Ob \mathcal{X}$, the canonical sequence (\ref{sqe}) splits.
\end{enumerate}
\end{definition}

\subsubsection{TTF triples}
In 1965, Jans \cite{Jan65}  first introduced the notion of torsion torsion-free theory which is now called a \emph{TTF triple}. In this subsection, we will recall the definitions of TTF triples and split TTF triples, as well as the TTF triples generated by a class of objects.

\begin{definition}
\begin{enumerate}
    \item A \emph{TTF triple} in an Abelian category $\mathcal{A}$ is a triple of full subcategories $(\mathcal{X}, \mathcal{Y}, \mathcal{Z})$ such that both $(\mathcal{X}, \mathcal{Y})$ and $(\mathcal{Y}, \mathcal{Z})$ are torsion pairs. 
    \item A TTF triple $(\mathcal{X}, \mathcal{Y}, \mathcal{Z})$ is said to \emph{split} if both torsion pairs $(\mathcal{X}, \mathcal{Y})$ and $(\mathcal{Y}, \mathcal{Z})$ split.
\end{enumerate}   
\end{definition}

\begin{definition} (\cite[Definition 4.6]{PSV21})
Let $(\mathcal{X}, \mathcal{Y}, \mathcal{Z})$ be a TTF triple in $(\mathcal{A}^{\op}, \Ab)$ and let $\mathcal{S}$ be a class of right $\mathcal{A}$-modules. We say that $(\mathcal{X}, \mathcal{Y}, \mathcal{Z})$ is generated by $\mathcal{S}$ when the torsion pair $(\mathcal{X}, \mathcal{Y})$ is \emph{generated by $\mathcal{S}$}, i.e., when $ \mathcal{Y}=\mathcal{S}^{\perp}$ (the meaning of $(-)^{\perp}$ here is the same as that of Definition \ref{torpairdef}). Furthermore, we say that $(\mathcal{X}, \mathcal{Y}, \mathcal{Z})$ is \emph{generated by finitely generated projective $\mathcal{A}$-modules} when it is generated by a set of finitely generated projective objects of $(\mathcal{A}^{\op}, \Ab)$.    
\end{definition}

\subsubsection{Abelian recollements}
There is an alternative way to think about the TTF triples, namely, the Abelian recollements. Recollements were first introduced in the context of triangulated categories by Beilinson, Bernstein and Deligne \cite{BBD82}. A fundamental example of a recollement of Abelian categories appeared in the construction of perverse sheaves by MacPherson and Vilonen \cite{MV86}. Now, for the convenience of the readers, we will record the definition of the Abelian recollement of $(\mathcal{A}^{\op}, \Ab)$ below.

\begin{definition}
 Let $\mathcal{A}$ be a small preadditive category. A \emph{recollement} $\mathscr{R}$ of $(\mathcal{A}^{\op}, \Ab)$ by Abelian categories $\mathcal{X}$ and $\mathcal{Y}$ (also called an \emph{Abelian recollement}) is a diagram of additive functors
$$
\xymatrix @C=4pc {
\mathscr{R}:\ \mathcal{Y} \ar[rr]|(.5){i_*} & & (\mathcal{A}^{\op}, \Ab) \ar@/^1pc/[ll]^(0.5){i^!} \ar@/_1pc/[ll]_(0.5){i^*}  \ar[rr]|(.5){j^*} & & \mathcal{X} \ar@/^1pc/[ll]^(0.5){j_*} \ar@/_1pc/[ll]_(0.5){j_!}
} 
$$ 
satisfying the following conditions:
\begin{enumerate}
    \item $(i^*, i_*, i^!)$ and $(j_!, j^*, j_*)$ are adjoint triples;
    \item the functors $i_*$, $j_!$ and $j_*$ are fully faithful;
    \item ${\rm Im}(i_*)={\rm Ker}(j^*)$.
\end{enumerate}
\end{definition}

Two Abelian recollements $\mathscr{R}: (\mathcal{Y}, (\mathcal{A}^{\op}, \Ab), \mathcal{X})$ and $\mathscr{R}': (\mathcal{Y}', (\mathcal{A}^{\op}, \Ab), \mathcal{X}')$ of $(\mathcal{A}^{\op}, \Ab)$ are said to be \emph{equivalent} if there are equivalences $\Phi: (\mathcal{A}^{\op}, \Ab) \to (\mathcal{A}^{\op}, \Ab)$ and $\Psi: \mathcal{X} \to \mathcal{X}'$ such that the following diagram commutes, up to natural isomorphism:
$$
\xymatrix @C=8pc {
(\mathcal{A}^{\op}, \Ab) \ar[r]^(.55){j^*} \ar[d]_{\Phi} & \mathcal{X} \ar[d]^{\Psi} \\
(\mathcal{A}^{\op}, \Ab) \ar[r]^(.55){(j^*)'} & \mathcal{X}'
} 
$$

\section{The characterization of modules over dg-representations of a small category} \label{cha}
In this section, a characterization of the category of modules $\rMod R$ over a dg-representation of a small category will be given. More specifically, we will characterize the category $\rMod R$ as an ordinary category of dg-modules over $Gr(R)$. By a theorem of Freyd, we can also characterize $\rMod R$ as an Ab-valued functor category. 

\begin{theorem} \label{howeplus}
Let $\calC$ be a small category and let $R: \calC \to \dgCat_k$ be a dg-representation of the category $\calC$, then we have the following equivalence
 $$
 \rMod R \simeq  (Gr(R)^{\op}, \dgCh(k)).
 $$    
\end{theorem}

\begin{proof}
Let $M \in \rMod R$ and $F \in (Gr(R)^{\op}, \dgCh(k))$, we define
$$
 \Phi: \rMod R \to (Gr(R)^{\op}, \dgCh(k))
$$ 
and
$$
 \Psi: (Gr(R)^{\op}, \dgCh(k)) \to \rMod R
$$ 
by $\Phi(M)=F_M$ and $\Psi(F)=M_F$ respectively, where $F_M$ and $M_F$ are defined as follows. 

Given $M \in \rMod R$, let's define $F_M$ below:
$$
\xymatrix{
    Gr(R)^{\op} \ar[rr]^{F_M} &  & \dgCh(k) &  \\
   _ix  \ar@{}[u]|{\begin{sideways}$\in$\end{sideways}} \ar[dd]_{(\cdots,0,f_a^\bullet,0,\cdots)} & & M_i(x) \ar@{}[u]|{\begin{turn}{90}$\in$\end{turn}} & \\
     & \longmapsto & & M_j(R(a)x) \ar[ul]_{M(a)_x} \\
   _jy  & & M_j(y) \ar[uu]|{M(a)_x \circ M_j(f_a^\bullet)} \ar[ur]_{M_j(f_a^\bullet)} & \\
     } 
$$
where $f_a^\bullet \in R(j)(R(a)x, y)$ is a map in dg-category $R(j)$, and $M(a)_x: a^*M_j(x) \to M_i(x)$ is the natural transformation $M(a): a^*M_j \to M_i$ at object $x$. For a general morphism $(f_a^p)_{a, p} \in \Hom_{Gr(R)}(_ix,_jy)$, we define
$$
F_M[(f_a^p)_{a,p}]:=\sum_{a,f_a^\bullet}M(a)_x \circ M_j(f_a^\bullet).
$$

Given $F \in (Gr(R)^{\op}, \dgCh(k))$, let's define $M_F$ as follows:
$$
\xymatrix{
   i \in \Ob \calC & \overset{M_F}{\longmapsto} & (R(i)^{\op} \ar[rr]^{M_i} & & \dgCh(k)) \\
   &  & x \ar@{}[u]|{\begin{sideways}$\in$\end{sideways}} \ar[dd]_{g^\bullet} & & F(_ix) \ar@{}[u]|{\begin{sideways}$\in$\end{sideways}} \\
   & & & \longmapsto  & \\
   & & x' & & F(_ix') \ar[uu]_{F[(\cdots,0,g^\bullet \circ (\eta_ix)^\bullet,0,\cdots)]} \\
     } 
$$
where $g^\bullet \circ (\eta_ix)^\bullet$ is the composite of 
$$
R(1_i)x \overset{(\eta_ix)^\bullet}{\longrightarrow} 1_{R(i)}x=x \overset{g^\bullet}{\longrightarrow} x'.
$$
and $(\cdots,0,g^\bullet \circ (\eta_ix)^\bullet,0,\cdots)$ is a morphism of $\Hom_{Gr(R)}(_ix,_ix')$.
Recall that $a^*: (R(j)^{\op}, \dgCh(k)) \to (R(i)^{\op}, \dgCh(k))$ is the restriction functor along $R(a): R(i) \to R(j)$ (see Remark \ref{diff} $(2)$), we define $M(a): a^*M_j \to M_i$ as follows:
$$
\xymatrix @C=8pc {
   a^*M_j(x) \ar[r]^-{M(a)_x} \ar @{=} [d] & M_i(x) \ar @{=} [dd] \\
   M_j(R(a)x) \ar @{=} [d] &  \\
  F(_jR(a)x) \ar[r]^-{F[(\cdots,0,(1_{R(a)x})^\bullet,0,\cdots)]}  & F(_ix) \\
     } 
$$
where $(\cdots,0,(1_{R(a)x})^\bullet,0,\cdots)$ is a morphism in $\Hom_{Gr(R)}(_ix, _jR(a)x)$. Thus we define $M(a)_x:=F[(\cdots,0,(1_{R(a)x})^\bullet,0,\cdots)]$, where $(1_{R(a)x})^\bullet$ is in the $a$-th component. Then, like the proof of \cite[Theorem 3.1.4]{Wu24a}, it is routine to check that the functors $\Phi$, $\Psi$ are both well defined, and they give rise to the equivalence as desired. This completes the proof. 
\end{proof}

\begin{remark}
The Theorem above can be viewed as the dg-version of the results \cite[Theorem A]{Wu24a} and \cite[Theorem 3.0.1]{Wu24} (or see \cite[Proposition 5]{How81}) with trivial topology.
\end{remark}

Let $D^n(k)$ be the chain complex which is $k$ in degree $n-1$ and $n$ and $0$ elsewhere. As a consequence of the Theorem \ref{howeplus} above, we immediately have the following corollary. 

\begin{corollary} \label{Groproj}
Let $\calC$ be a small category and let $R: \calC \to \dgCat_k$ be a dg-representation of the category $\calC$. Then the category of right $R$-modules $\rMod R$ is a Grothendieck category and has a projective generator.    
\end{corollary}

\begin{proof}
Since the category of chain complexes $\Ch(k)$ has a set of finitely generated projective generators $D^n(k)$ (see \cite[page 1658]{EG98}), then by \cite[Theorem 4.2]{AG16}, we have that the category $(Gr(R)^{\op}, \dgCh(k))$ is a Grothendieck category with a set of small projective generators $\{ G_{_ix,n}:=Gr(R)(-,_ix) \oslash D^n(k)~|~ _ix \in \Ob Gr(R), n \in \mathbb{Z}\}$ (one can also see \cite[Corollary 3.4]{GJ19}). So does $\rMod R$ by Theorem \ref{howeplus}. Here $\oslash$ is a functor, one can see \cite[Definition 2.7]{AG16} for its explicit definition. Let $G:=\bigoplus_{_ix, n} G_{_ix,n}$, then $G$ is a projective generator of $\rMod R$.
\end{proof}

\begin{remark}
\begin{enumerate}
    \item This result can be seen as the dg-version of \cite[Theorem A]{Wu24a} (or \cite[Theorem 3.18]{EV17}).
    \item The projective generator $G$ in (the proof of) the Corollary \ref{Groproj} above is \emph{not} small even when $Gr(R)$ is object-finite, thus we \emph{can not} establish similar result as \cite[Theorem C]{Wu24a} and \cite[Theorem A]{WX23}.
\end{enumerate}   
\end{remark}

Let's view the set $\{ G_{_ix,n}~|~ _ix \in \Ob Gr(R), n \in \mathbb{Z}\}$ of finitely generated projective generators as a full subcategory of $\rMod R$, and denote it by $\calP$. Then one can characterize $\rMod R$ as an additive Ab-valued functor category via a theorem of Freyd.

\begin{corollary} \label{redab}
Let $\calC$ be a small category and let $R: \calC \to \dgCat_k$ be a dg-representation of the category $\calC$. Then there is an equivalence 
$$
\rMod R \simeq (\calP^{\op}, \Ab).
$$    
\end{corollary}

\begin{proof}
By the proof of Corollary \ref{Groproj}, we know that $\{ G_{_ix,n}~|~ _ix \in \Ob Gr(R), n \in \mathbb{Z}\}$ is a set of small projective generators, then the equivalence in the corollary follows from \cite[Theorem 3.1]{Mit72}.   
\end{proof}

\section{Applications} \label{app}
In this section, as some applications of the results of the previous section, we will classify the hereditary torsion pairs, (split) TTF triples and Abelian recollements of the category $\rMod R$ of modules over a dg-representation of a small category. 

Our first result of this section indicates that the hereditary torsion pairs in the category of modules over a dg-representation of a small category can be classified by the linear Grothendieck topologies or the Gabriel filters. 

\begin{theorem} \label{htp}
Let $\calC$ be a small category and let $R: \calC \to \dgCat_k$ be a dg-representation of the category $\calC$. Let $\calP$ be a full subcategory of $\rMod R$ with the set of finitely generated projective generators $\{ G_{_ix,n}~|~ _ix \in \Ob Gr(R), n \in \mathbb{Z}\}$ as objects. Then
\begin{enumerate}
    \item There is an (explicit) one-to-one correspondence between linear Grothendieck topologies on $\calP$ and hereditary torsion pairs in $\rMod R$, and
    \item There exists a bijection between hereditary torsion theories in $\rMod R$ and families of Gabriel filters $\mathcal{L}_{G_{_ix,n}}$, of subobjects of the generators $\{ G_{_ix,n}~|~ _ix \in \Ob Gr(R), n \in \mathbb{Z}\}$ provided that they satisfy the conditions $1)$-$3)$ and $4')$ as listed in \cite[Proposition 3.2 \& 3.5]{LLN92}.     
\end{enumerate}
\end{theorem}

\begin{proof}
\begin{enumerate}
    \item By Corollary \ref{redab}, we have
$$
\rMod R \simeq (\calP^{\op}, \Ab).
$$
Then, by \cite[Theorem 3.7]{PSV21}, there is an explicit one-to-one correspondence between the linear Grothendieck topologies on $\calP$ and the hereditary torsion pairs in $(\calP^{\op}, \Ab)$. 
     \item  By Corollary \ref{Groproj} and its proof, we know that $\rMod R$ is a Grothendieck category with a set of finitely generated projective generators $\{ G_{_ix,n} \}$. Then the statement $(2)$ follows from \cite[Proposition 3.4 \& 3.5]{LLN92}. 
     This completes the proof. 
\end{enumerate}
\end{proof}

Our second result of this section says that one can use idempotent ideals of $\calP$ to classify the TTF triples in $\rMod R$. 

\begin{theorem} \label{ttf}
Let $\calC$ be a small category and let $R: \calC \to \dgCat_k$ be a dg-representation of the category $\calC$. Let $\calP$ be a full subcategory of $\rMod R$ with the set of finitely generated projective generators $\{ G_{_ix,n}~|~ _ix \in \Ob Gr(R), n \in \mathbb{Z}\}$ as objects. Then there is an (explicit) one-to-one correspondence between idempotent ideals of $\calP$ and TTF triples in $\rMod R$. 
\end{theorem}

\begin{proof}
By Corollary \ref{redab}, we have
$$
\rMod R \simeq (\calP^{\op}, \Ab).
$$
 Using \cite[Theorem 4.5]{PSV21}, there is an explicit one-to-one correspondence between idempotent ideals of $\calP$ and TTF triples in $(\calP^{\op}, \Ab)$, so we are done.
\end{proof}

Our third result says that the split TTF triples in $\rMod R$ are one-to-one correspond to the idempotents of the center $Z(\mathcal{P})$ of $\mathcal{P}$.

\begin{theorem} \label{sttf}
Let $\calC$ be a small category and let $R: \calC \to \dgCat_k$ be a dg-representation of the category $\calC$. Let $\calP$ be a full subcategory of $\rMod R$ with the set of finitely generated projective generators $\{ G_{_ix,n}~|~ _ix \in \Ob Gr(R), n \in \mathbb{Z}\}$ as objects. Then there is an (explicit) one-to-one correspondence between idempotents of the center $Z(\calP)$ of $\calP$ and split TTF triples in $\rMod R$.      
\end{theorem}

\begin{proof}
Using Corollary \ref{redab} together with \cite[Theorem D and Proposition 4.13]{PSV21}, then we will obtain the result.  
\end{proof}

The last result is devoted to investigate the Abelian recollements in the category of $\rMod R$, and we obtain a classification of the Abelian recollements in $\rMod R$ as follows.  

\begin{theorem} \label{ar}
There are (explicit) one-to-one correspondences between:
\begin{enumerate}
    \item the equivalence classes of recollements of $\rMod R$ by module categories over small preadditive categories;
    \item the TTF triples in $\rMod R$ generated by finitely generated projective $\calP$-modules;
    \item the idempotent ideals of $\calP$ which are the trace of a set of finitely generated projective $\calP$-modules;
    \item the idempotent ideals of the additive closure $\widehat{\calP}$ of $\calP$ generated by a set of idempotent endomorphisms.
    \item the full subcategory of the Cauchy completion $\widehat{\calP}_\oplus$ which are closed under coproducts and summands. 
\end{enumerate}
\end{theorem}

\begin{proof}
These bijections follow from Corollary \ref{redab} and \cite[Theorem 4.8]{PSV21}. 
\end{proof}

\section*{Acknowledgments}
I would like to thank my advisor Prof. Fei Xu \begin{CJK*}{UTF8}{}
\CJKtilde \CJKfamily{gbsn}(徐斐) \end{CJK*} in Shantou University for his guidance.

\bibliographystyle{plain}
\bibliography{ref}
\end{document}